\documentclass[12pt]{article}

\usepackage{amsmath}
\usepackage{amsfonts}
\usepackage{verbatim}
\usepackage{amsthm}
\usepackage{amssymb}
\usepackage{epsfig}
\usepackage{scrextend}
\usepackage{float}

\newtheorem{theorem}{Theorem}[section]
\newtheorem{proposition}[theorem]{Proposition}
\newtheorem{lemma}[theorem]{Lemma}

\theoremstyle{definition}

\newcommand{\mc}[0]{\mathcal}

\newcommand{\dual}[0]{^*}

\newcommand{\bb}[0]{\mathbb}
\newcommand{\bs}[0]{\backslash}

\begin{document}

\title{The family of bicircular matroids closed under duality}

\author{Vaidy Sivaraman\footnote{Department of Mathematics and Statistics, Mississippi State University, Mississippi State, MS 39762, USA. \underline{Email} vaidysivaraman@gmail.com}\, and Daniel Slilaty\footnote{Department of Mathematics and Statistics, Wright State University, Dayton, OH 45435, USA. \underline{Email} daniel.slilaty@wright.edu.}}

\maketitle

\abstract{We characterize the 3-connected members of the intersection of the class of bicircular and cobicircular matroids. Aside from some exceptional matroids with rank and corank at most 5, this class consists of just the free swirls and their minors.}

\section{Introduction}

Whitney showed that the intersection of the classes of graphic and cographic matroids is exactly the class of planar graphic matroids. Slilaty showed \cite{Slilaty:Cographic} that the intersection of the classes of connected cographic matroids and connected signed-graphic matroids is exactly the class of connected cographic matroids of projective-planar graphs. Carmesin \cite{carmesin:Whitney} greatly extends these ideas by defining a class $r$-locally planar graphs and a class of $r$-local matroids which describes the intersection of this class of matroids with the class of cographic matroids. Carmesin also extends these ideas to 2-complexes embedded in 3-space \cite{carmesin:III,carmesin:IV,carmesin:I,carmesin:II,carmesin:V}. The intersections investigated in all these works are described in terms of topological embeddings.

In this paper, we determine the intersection of the classes of 3-connected bicircular matroids and cobicircular matroids and find that it, unsurprisingly, is not described in terms of topological embeddings. Aside from some exceptional matroids with rank and corank at most 5, this intersection consists solely of the free swirls and their minors. The free swirl is the bicircular matroid $B(2C_n)$ in which $2C_n$ is the graph obtained from the cycle of length $n$ by doubling each edge.

The class of bicircular matroids is, of course, a minor-closed class of matroids that is properly contained within the class of transversal matroids. Transversal matroids are again not closed under duality but also not closed under taking minors, either. In order to describe the minors and duals of transversal matroids, the more general class of gammoids is used. Other investigations of natural subclasses of transversal matroids have also found closure under both minors and duality to be a property of interest. Bonin and de Mier \cite{BoninDeMier:LatticePathEnumerative, BoninDeMier:LatticePathStructure} investigated the class of lattice-path matroids and found that it is both minor closed and duality closed. Bonin and Gim\'{e}nez \cite{BoninGimenez:MultiPath} investigate the class of multi-path matroids which sits properly between the classes of lattice-path and transversal matroids and yet is still closed under both minors and duality. Las Vergnas \cite{LasVergnas} showed that the class of fundamental transversal matroids is closed under duality but not closed under minors. Brualdi \cite{Brualdi} shows that Las Vergnas' result is a corollary of a more general phenomenon. Neudauer \cite{Neudauer} characterizes the intersection of the classes of bicircular matroids and fundamental transversal matroids. Her result is thus related to ours. The intersection of the class of bicircular matroids with lattice-path matroids and multi-path matroids may be an interesting topic of investigation.

In the remainder of this introduction, we describe our main result which is Theorem \ref{T:MainResult}. Theorem \ref{T:MainResult} is actually a corollary of a stronger statement (Theorem \ref{T:StrongerThenMain}) which is stated and proven in the final section of the paper. Theorem \ref{T:StrongerThenMain} also contains information about excluded minors. Given a matroid $M$ we say that $M$ is \emph{bicircular} when $M=B(G)$ for some graph $G$, $M$ is \emph{cobicircular} when $M\dual=B(G)$ for some graph $G$, and $M$ is doubly bicircular when it is both bicircular and cobicircular. Consider the following six families of 3-connected doubly bicircular matroids, only one of which is infinite. The first is the collection of free swirls and their minors. Again, the free swirl of rank $n$ is identically self dual and is the bicircular matroid $B(2C_n)$ where $2C_n$ is the graph obtained from the cycle of length $n$ with each edge doubled.

\begin{figure}[H]
\begin{center}
\includegraphics[page=5,scale=0.6]{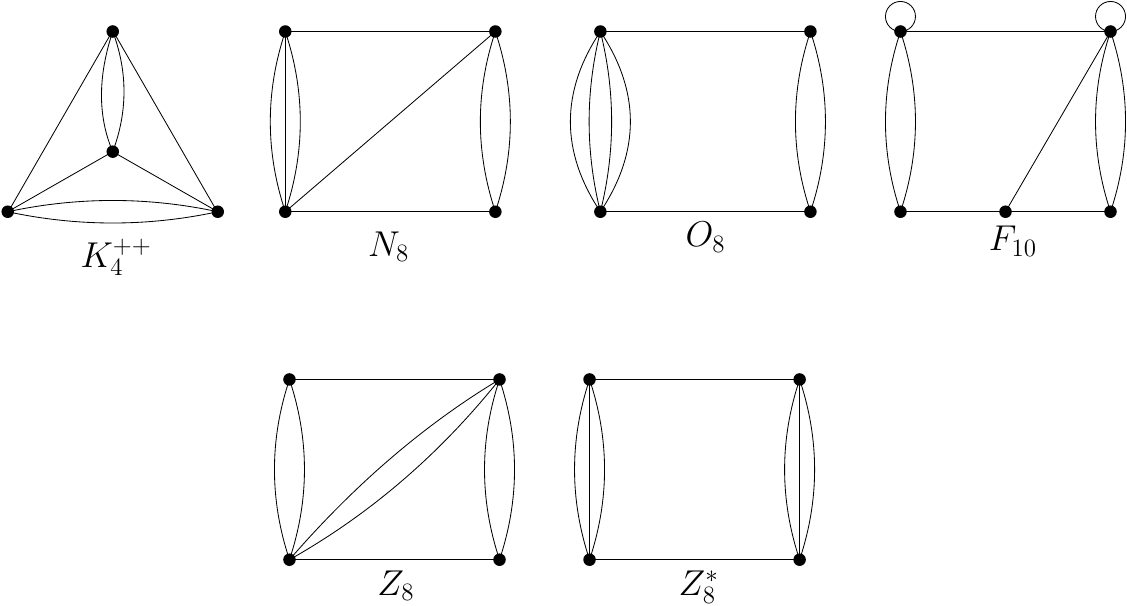}
\end{center}
\caption{Some graphs for doubly bicircular matroids.} \label{F:SporadicFamilies}
\end{figure}

\noindent The remaining five families of 3-connected doubly bicircular matroids are all finite and have rank and corank at most 5. They are the bicircular matroids of the graphs of Figure \ref{F:SporadicFamilies} and their minors. The reader can check that $B(K_4^{++})$, $B(N_8)$, $B(O_8)$, and $B(F_{10})$ are all self dual and that $B\dual(Z_8)\cong B(Z_8\dual)$. This check can be done by hand or by using the SageMath software package. One way to represent the bicircular matroid of a graph $G$ in SageMath is as follows. Define a $\bb Q$-matrix $A$ whose rows are indexed by $V(G)$ and whose columns are indexed by $E(G)$. The column corresponding to a link $e$ having endpoints in rows $i$ and $j$ should have a $-1$ in row $i$, a prime number $p_e$ unique to $e$ in row $j$, and zeros in all other rows. The column corresponding to a loop incident to vertex $v$ should be the elementary column vector corresponding to the row for $v$. Now $M(A)=B(G)$.

\begin{theorem}[Main Result] \label{T:MainResult}
If $M$ is a 3-connected matroid, then $M$ is doubly bicircular if and only if $M$ is a minor of a free swirl or $M$ is a minor of the bicircular matroid of one of the graphs in Figure \ref{F:SporadicFamilies}.
\end{theorem}

\section{Preliminaries}

We assume that the reader is familiar with matroid theory as in Oxley's book \cite{Oxley:2ndEdition}. We will, however, briefly review the definition of a bicircular matroid and modify it slightly so that the class of bicircular matroids is closed under taking minors.

A graph $G$ consists of a vertex set $V(G)$ and an edge set $E(G)$ in which an edge $e\in E(G)$ is either a \emph{link} connecting two distinct vertices, a \emph{loop} on a single vertex, or a \emph{free edge} which is not incident to any vertex.
Given a graph $G$, the bicircular matroid $B(G)$ has element set $E(G)$ in which the free edges are matroid loops in $B(G)$ and every other circuit of $B(G)$ is the edge set of a subgraph of $G$ which is a subdivision of one of the graphs shown in Figure \ref{F:BicircularCircuits}. For a matroid $M$, when $M=B(G)$ we say that $G$ is a \emph{bicircular representation} of $M$.

\begin{figure}[H]
\begin{center}
\includegraphics[page=1,scale=1]{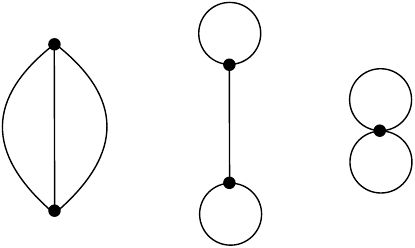}
\end{center}
\caption{Graphs describing the circuits of $B(G)$.} \label{F:BicircularCircuits}
\end{figure}

For any edge $e$, $B(G)\bs e=B(G\bs e)$. If $e$ is a free edge, then $B(G)/e=B(G)\bs e$. If $e$ is a link, then $B(G)/e=B(G/e)$. If $e$ is a loop which is incident to vertex $v$, then $B(G)/e=B(\tilde G)$ in which $\tilde G$ is defined a follows: $V(\tilde G)=V(G)-v$, $E(\tilde G)=E(G)-e$, if $x\neq e$ is a loop in $G$ that is incident to $v$ then $x$ becomes a free edge in $\tilde G$, if $x$ is a link in $G$ that is incident to $v$ then $x$ is a loop in $\tilde G$ which will be incident to its second endpoint from $G$, and all other $e\in E(G)$ remain as they are in $G$.

The usual notion of a graph $H$ being a minor of a graph $G$ does not include the contraction operation for loops in the previous paragraph. Rather loop contractions in graphs coincide with loop deletions. In particular, if graph $H$ is a minor of a graph $G$, then $B(H)$ is a minor of $B(G)$; however, if matroid $B(H)$ is a minor of matroid $B(G)$, then it need not be the case that $H$ is a minor of $G$ in the graphic sense.

We use Proposition \ref{P:Wagner3Biconnectivity} without further mention.

\begin{proposition}[Wagner {\cite[Prop.2]{Wagner:ConnectivityBicircular}}] \label{P:Wagner3Biconnectivity}
If $G$ is connected, has no free edges and has at least three vertices, then $B(G)$ is 3-connected if and only if $G$ is 2-connected, has no degree-2 vertices, and has no two loops incident to the same vertex.
\end{proposition}

Given an integer $n\geq 2$, an $n$-\emph{multilink} is the graph consisting of two vertices along with $n$ links connecting them. The $n$-multilink is denoted by $nK_2$. A graph is \emph{separable} if there exists subgraphs $G_1$ and $G_2$ for which $G=G_1\cup G_2$ but $G_1\cap G_2$ is either empty or a single vertex. Thus a graph on at least three vertices is non-separable if and only if it is 2-connected and loopless. In this paper we will make use of the \emph{canonical tree decomposition} of a non-separable graph $G$. For information
on the canonical tree decomposition see one of \cite{CunninghamEdmonds:Decompositions},
\cite[pp.308--315]{Oxley:2ndEdition}, or \cite{Tutte:ConnectivityInGraphs} for a full description. In short, if $G$ is
non-separable, then there is a unique labeled tree $T$ satisfying the following.\begin{itemize}
\item Each vertex $v$ in $T$ is labeled with either a 3-connected simple graph, a cycle of length at least three, or
$mK_2$ for some $m\geq 3$.
\item No two cycle-labeled vertices are adjacent in $T$ and no two multilink-labeled vertices are adjacent in $T$.
\item If $e$ is an edge of $T$ whose endpoints are labeled with graphs $G_1$ and $G_2$, then $e$ corresponds to an
    edge $e_i$ in $G_i$.
\item $G$ is obtained by executing the 2-sums indicated by the vertex labels of $T$ along the edges
    indicated by the edges of $T$.\end{itemize} An important consequence of this tree decomposition is that if $T_0$ is
    a subtree of $T$, then the graph $G_0$ obtained by executing the 2-sums indicated in $T_0$ is a minor of $G$.

\section{Excluded Minors} \label{S:ExcludedMinors}

In this section we show that the bicircular matroids Theorem \ref{T:StrongerThenMain} Part (2) are minor-minimal matroids that are not cobicircular matroids.

\begin{figure}[H]
\begin{center}
\includegraphics[page=8,scale=0.6]{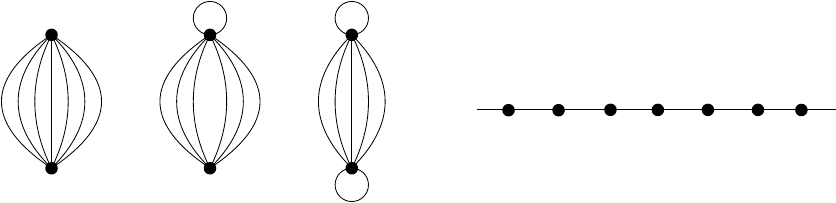}
\end{center}
\caption{The matroid $U_{2,7}$ and its bicircular representations.} \label{F:L7}
\end{figure}

\begin{proposition} \label{P:L7}
The matroid $U_{2,7}$ is bicircular but minimally not cobicircular. Figure \ref{F:L7} shows all possible bicircular representations of $U_{2,7}$.
\end{proposition}
\begin{proof}
It is evident that the graphs in Figure \ref{F:L7} are all possible bicircular representations $U_{2,7}$.  Matthews' Theorem 3.8 from \cite{Matthews:Bicircular} implies that  $U_{2,7}$ is minimally not cobicircular.
\end{proof}

Given an element $e$ in a matroid $M$ with a transitive symmetry group, let $M'$ be the matroid obtained from $M$ by replacing $e$ with a pair of coparallel elements. When we write $M''$ we mean that $M$ has a 2-transitive symmetry group and two distinct elements of $M$ are replaced by coparallel pairs.

\begin{figure}[H]
\begin{center}
\includegraphics[page=9,scale=0.5]{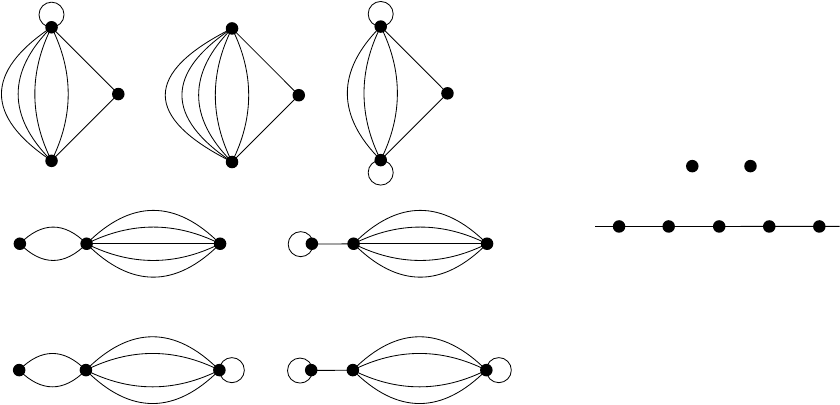}
\end{center}
\caption{The matroid $U_{2,6}'$ along with its bicircular representations.} \label{F:L6'}
\end{figure}

\begin{figure}[H]
\begin{center}
\includegraphics[page=16,scale=0.5]{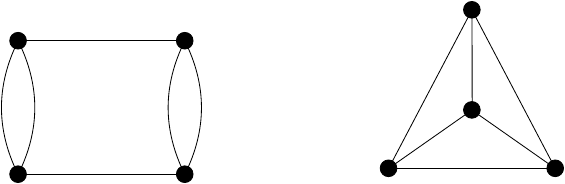}
\end{center}
\caption{The graphs $C_4^{++}$ and $K_4$ are the only bicircular representations of $U_{4,6}$.} \label{F:U46}
\end{figure}

\begin{proposition} \label{P:ParallelCoparallelElements}
Let $B$ be a doubly bicircular matroid and $e\in E(B)$. If for every bicircular representation $G$ of $B\dual$ the edge $e$ is a link in $G$, then $B'$ is bicircular but not cobicircular.
\end{proposition}
\begin{proof}
A potential bicircular representation of $(B')\dual$ would be obtained from a bicircular representation $G$ of $B\dual$ by adding an edge, call it $f$. However, since $e$ is a link in $G$, there is no way that $e$ and $f$ could be parallel elements in $B(G\cup f)$.
\end{proof}

\begin{proposition} \label{P:L6'}
The matroid $U_{2,6}'$ is bicircular but minimally not cobicircular. Figure \ref{F:L6'} shows all possible bicircular representations of $U_{2,6}'$ and Figure \ref{F:U46} shows all possible bicircular representations of $U_{4,6}$.
\end{proposition}
\begin{proof}
That the graphs in Figure \ref{F:L6'} are the complete list all possible bicircular representations of $U_{2,6}'$ is clear. A bicircular representation $G$ of $U_{2,6}\dual=U_{4,6}$ must have four vertices of degree at least 3 each. Since $G$ has 6 edges, we now get that all vertices have degree 3.  All possible 2-connected cubic graphs on four vertices are shown in Figure \ref{F:L6'} and both are bicircular representations of $U_{4,6}$. Since there is no bicircular representation of $U_{4,6}$ which uses loops, Proposition \ref{P:ParallelCoparallelElements} implies that $U_{2,6}'$ is not cobicircular. The reader can check that $U_{2,6}'$ is minimally not cobicircular.
\end{proof}

\begin{figure}[H]
\begin{center}
\includegraphics[page=6,scale=0.6]{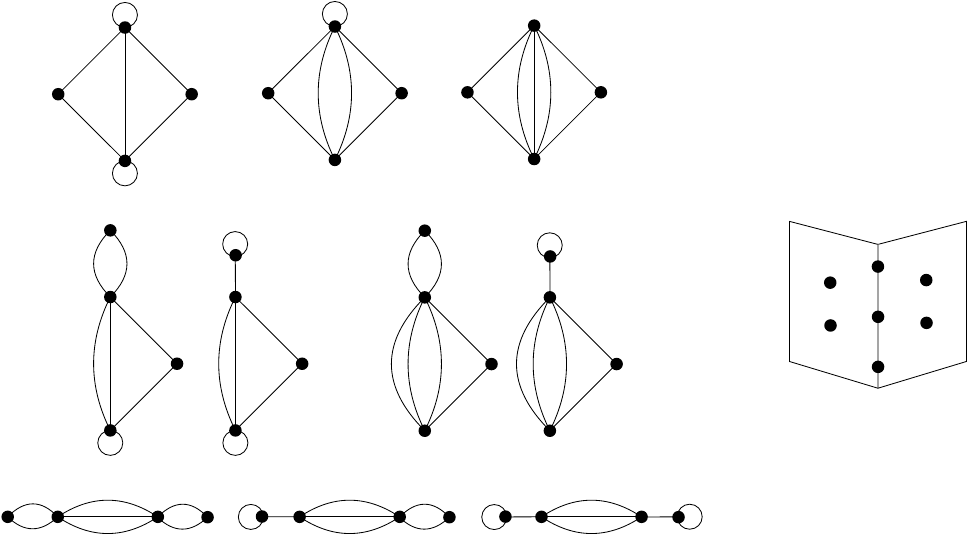}
\end{center}
\caption{The matroid $U_{2,5}''$ and its bicircular representations.} \label{F:L5''}
\end{figure}

\begin{figure}[H]
\begin{center}
\includegraphics[page=10,scale=0.6]{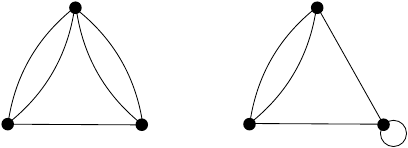}
\end{center}
\caption{The bicircular representations of $U_{3,5}$.} \label{F:U35}
\end{figure}

\begin{proposition} \label{P:L5''}
The matroid $U_{2,5}''$ is bicircular but minimally not cobocircular. Figure \ref{F:L5''} shows all possible bicircular representations of $U_{2,5}''$. Figure \ref{F:U35} shows all possible bicircular representations of $U_{3,5}$.
\end{proposition}
\begin{proof}
That the graphs shown are all possible bicircular representations of $U_{2,5}''$ and $U_{3,5}$ can easily be checked by the reader. Since $U_{2,5}\dual=U_{3,5}$ and there is at most one loop in a bicircular representation of $U_{3,5}$ it follows that  $U_{2,5}''$ is not cobicircular as in the proof of Proposition \ref{P:ParallelCoparallelElements}. That $U_{2,5}''$ is minimally not cobicircular can be checked by the reader.
\end{proof}

\begin{figure}[H]
\begin{center}
\includegraphics[page=11,scale=0.8]{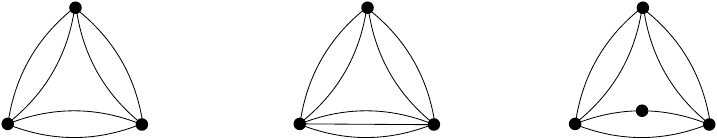}
\end{center}
\caption{The only bicircular representation of $U_{3,6}$ is $T_{2,2,2}$ \cite[Lemma 2.12]{Zaslavsky:BG2}. The matroids $B(T_{3,2,2})$ and $B(T_{2,2,2}')$ are both minimally not cobicircular.} \label{F:U36}
\end{figure}

\begin{proposition} \label{P:T223}
The only bicircular representation of $U_{3,6}$ is $T_{2,2,2}$. The matroids $B(T_{3,2,2})$ and $B(T_{2,2,2}')$ are both minimally not cobicircular.
\end{proposition}
\begin{proof}
That $T_{2,2,2}$ is the only bicircular representation of $U_{3,6}$ was noted in \cite[Lemma 2.12]{Zaslavsky:BG2}. Since $U_{3,6}\dual\cong U_{3,6}$, we get that $B(T_{2,2,2}')$ is not cobicircular by Proposition \ref{P:ParallelCoparallelElements}. The reader may check for minimality. It is evident that every single-element contraction and deletion of $B(T_{3,2,2})$ is cobicircular. Now consider and edge $e$ such that $B(T_{3,2,2})\bs e=B(T_{2,2,2})$. So now $B\dual(T_{3,2,2})/e\cong B(T_{2,2,2})$. Thus if we assume that $B(G)=B\dual(T_{3,2,2})$, then $B(G)$ would be forced to have two vertices of degree 3. These two degree-3 vertices in $G$ form cotriangles $B(G)=B\dual(T_{3,2,2})$ and so form triangles in $B(T_{2,2,3})$; however, $B(T_{2,2,3})$ contains only one triangle, a contradiction.
\end{proof}

Given a graph $G$, we let $G^\ell$ be the graph obtained from $G$ by adding a loop to some
vertex. If $G$ is a loopless graph, then $G^\circ$ is the graph obtained from $G$ by adding a loop at each vertex.

\begin{proposition} \label{P:3Whirl}
The only bicircular representation of $\mc W^3$ (i.e., the rank-3 whirl) is $C_3^\circ$.
\end{proposition}
\begin{proof}
Let $\{1,2,3,4,5,6\}$ be the groundset of $\mc W^3$ and let $G$ be a bicircular representation of $\mc W^3$. Consider the two triangles $\{1,2,3\}$ and $\{1,4,5\}$ in $\mc W^3$. Evidently there are three possibilities for $G\bs 6$. They are shown in the first row of Figure \ref{F:3Whirl}.

\begin{figure}[H]
\begin{center}
\includegraphics[page=17,scale=0.7]{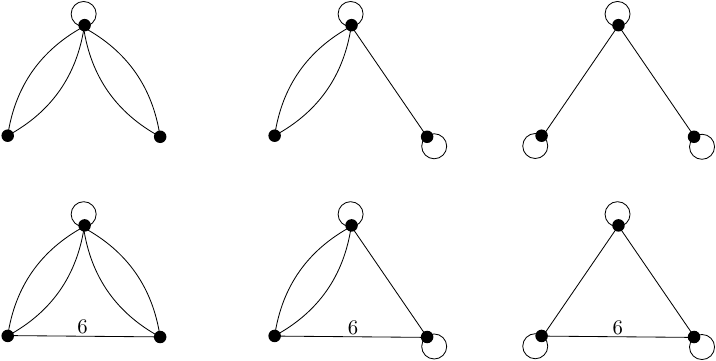}
\end{center}
\caption{The sixth graph shown is the only bicircular representation of $\mc W^3$.} \label{F:3Whirl}
\end{figure}

Since $\mc W^3$ is 3-connected, $G$ must be 2-connected and so edge 6 must be a link bridging the cut vertex of $G\bs 6$. Thus the three possibilities for $G$ are shown in the second row of Figure \ref{F:3Whirl}; however, only the sixth graph is actually a bicircular representation for $\mc W^3$, as required.
\end{proof}

\begin{proposition} \label{P:ExcludedMinors3ConnectedGraphs}
The matroids $B(W_4)$ and $B(K_4^\ell)$ are minimally not cobicircular.
\end{proposition}
\begin{proof}
Note that $B(K_4^\ell)/\ell$ is the rank-3 whirl. Since $(\mc W^3)\dual\cong \mc W^3$, the only possible bicircular representation for $B\dual(K_4^\ell)$ would be the graph $G$ obtained from $C_3^\circ$ with $\ell$ added as a link. However, $B(G)$ has six triangles while $B(K_4^\ell)$ has only three cotriangles, a contradiction. We leave it to the reader to check minimality.

Consider the 4-wheel $W_4$ and let $e$ and $f$ be non-adjacent edges on the rim. Note that $W_4/\{e,f\}\cong T_{2,2,2}$. Thus $B\dual(W_4)\bs\{e,f\}=B\dual(W_4/\{e,f\})\cong U_{3,6}\dual\cong U_{3,6}\cong B(T_{2,2,2})$. Since $B(T_{2,2,3})$ is not cobicircular (Proposition \ref{P:T223}) the only possible bicircular representation of $B\dual(W_4)$ would be the graph $G$ obtained from $T_{2,2,2}$ with $e$ and $f$ added as loops at two different vertices. Hence $G=2C_5/\{x,y\}$ where $x$ and $y$ are non-adjacent links. Thus $B\dual(G)\cong B(2C_5\bs\{x,y\})\cong B(W_4)$; however, this contradicts the result of Wagner that the wheels $W_n$ with $n\geq4$ are the unique bicircular representations of their bicircular matroids \cite[Proposition 5]{Wagner:ConnectivityBicircular}. We leave it to the reader to check minimality.
\end{proof}

\begin{figure}[H]
\begin{center}
\includegraphics[page=7,scale=0.7]{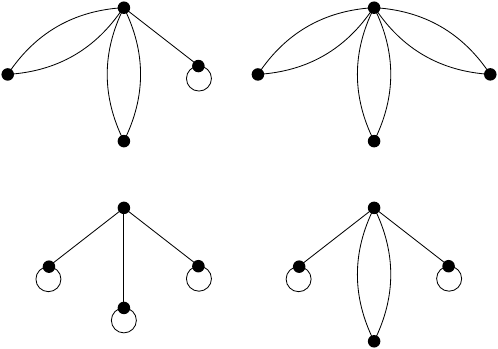}
\end{center}
\caption{These four graphs are all possible bicircular representations of the graphic matroid $M(K_{2,3})$.} \label{F:K23}
\end{figure}

The proof of Proposition \ref{P:K23} is left to the reader.

\begin{proposition} \label{P:K23}
The only bicircular representations of the graphic matroid $M(K_{2,3})$ are those shown in Figure \ref{F:K23}. The matroid $M(K_{2,3})$ is minimally not cobicircular.
\end{proposition}

\begin{figure}[H]
\begin{center}
\includegraphics[page=12,scale=0.8]{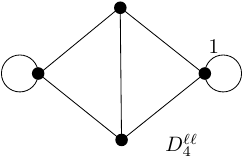}
\end{center}
\caption{The bicircular matroid of the graph $D_4^{\ell\ell}$ is minimally not cobicircular.} \label{F:LastExcludedMinor}
\end{figure}

\begin{proposition} \label{P:LastExcludedMinor}
The bicircular matroid $B(D_4^{\ell\ell})$ (see Figure \ref{F:LastExcludedMinor}) is minimally not cobicircular.
\end{proposition}
\begin{proof}
Note that $B(D_4^{\ell\ell})/1\cong\mc W^3$. If we assume by way of contradiction that $H$ is a bicircular representation of $B\dual(D_4^{\ell\ell})$, then $B(H\bs1)\cong\mc W^3$ and so $H\bs1\cong C_3^\circ$ by Proposition \ref{P:3Whirl}. Furthermore, edge $1$ cannot be added as a loop to $C_3^\circ$ because $B(D_4^{\ell\ell})$ is cosimple. Thus the only possibility for $H$ is $C_3^\circ$ with one link doubled. Edge 1 is now in three triangles of $B(H)$; however, edge 1 is only in one cotriangle of $B(D_4^{\ell\ell})$, a contradiction. We leave it to the reader to check minimality.
\end{proof}

\section{Proof of our main result}

\begin{proposition} \label{P:3Connected}
If $G$ is a 3-connected graph, then the following are equivalent.\begin{itemize}
\item[(1)] $B(G)$ is cobicircular.
\item[(2)] $G$ has no $W_4$-, $K_4^\ell$-, or $T_{2,2,3}$-minor.
\item[(3)] $G$ is a minor of $K_4^{++}$.
\end{itemize}
\end{proposition}
\begin{proof}
In the introduction we addressed the fact that $B(K_4^{++})$ is self dual and hence doubly bicircular. This proves $3\rightarrow1$. Propositions \ref{P:T223} and \ref{P:ExcludedMinors3ConnectedGraphs} prove $1\rightarrow2$. We now finish by proving $2\rightarrow3$.

 Let $G$ be a 3-connected graph which contains none of $W_4$, $K_4^\ell$, and $T_{2,2,3}$ as a minor. By Tutte's Wheel Theorem and the fact that there is no $K_4^\ell$-minor in $G$, $G$ is obtained from $K_4$ by a sequence of adding and de-contracting links while maintaining graph 3-connectedness at each step; furthermore, the first step must be adding a link, call it $e$, to $K_4$ to obtain the graph $K_4^+$. The second step cannot be decontraction because the only 3-connected graph that is a decontraction of $K_4^+$ is $W_4$. Thus the second step is adding another link to $K_4^+$, call it $f$. If $f$ is parallel to $e$, then $G$ contains a $T_{2,2,3}$-minor, a contradiction. If $f$ is adjacent to $e$ but not parallel to $e$, then again $G$ contains a $T_{2,2,3}$-minor, a contradiction. Thus adding $f$ to $K_4^+$ yields $K_4^{++}$.

The third step cannot be a decontraction because, again, a 3-connected decontraction of $K_4^{++}$ contains a  $W_4$-minor. The third step also cannot be adding a link because wherever a link is $K_4^{++}$ we obtain a graph with a $T_{2,2,3}$-minor, a contradiction. Thus $G$ is a minor of $K_4^{++}$.
\end{proof}

\begin{theorem} \label{T:StrongerThenMain}
If $B(G)$ is 3-connected, then the following are equivalent.
\begin{itemize}
\item[(1)] $B(G)$ is cobicircular.
\item[(2)] $G$ has no ordinary graph minor $H$ which is a bicircular representations of $U_{2,7}$, $U_{2,6}'$, $U_{2,5}''$, $M(K_{2,3})$, $B(T_{2,2,3})$ $B(T_{2,2,2}')$, $B(W_4)$, $B(K_4^\ell)$, and $B(D_4^{\ell\ell})$.
\item[(3)] $B(G)$ is a minor of a free swirl or is a minor of the bicircular matroid of some graph from Figure \ref{F:SporadicFamilies}.
\end{itemize}
\end{theorem}

\begin{lemma} \label{L:Maximality}
If $G\bs e$ is one of the graphs in Figure \ref{F:SporadicFamilies} and $B(G)$ is 3-connected, then $G$ contains one of the following graphs as a minor: $T_{2,2,3}$, $7K_2$, $6K_2^\ell$, and $5K_2^\circ$.
\end{lemma}
\begin{proof}
Since $B(G)$ is 3-connected, then $G$ is obtained from $G\bs e$ by adding $e$ as a link with both endpoints on the vertices of $G\bs e$ or as a loop that is not on the same vertex as an existing loop. If $G\bs e=K_4^{++}$ and $e$ is a loop, then $G$ has a $6K_2^\ell$-minor. If $G\bs e=K_4^{++}$ and $e$ is a link, then $G$ has a $7K_2$-minor. If $G\bs e\in\{N_8,O_8,Z_8,Z\dual_8\}$ and $e$ is a loop, then $G$ contains a $6K_2^\ell$-minor. If $G\bs e\in\{N_8,O_8,Z_8,Z\dual_8\}$ and $e$ is a link, then $G$ either has a $T_{2,2,3}$- or $7K_2$-minor. If $G\bs e=F_{10}$ and $e$ is a loop, then $G$ contains a $5K_2^\circ$-minor. If $G\bs e=F_{10}$ and $e$ is a link, then $G$ either has a $T_{2,2,3}$- or $6K_2^\ell$-minor.
\end{proof}

\begin{proof}[Proof of Theorem \ref{T:StrongerThenMain}]
In the introduction we addressed the fact that the free swirls and the bicircular matroids of the graphs in Figure \ref{F:SporadicFamilies} are actually doubly bicircular. This proves $3\rightarrow1$. In Section \ref{S:ExcludedMinors} we showed that the nine bicircular matroids listed in (2) are not cobicircular. This proves $1\rightarrow2$. We now finish by proving $2\rightarrow3$.

Because $B(G)$ is 3-connected, $G$ has at least two vertices. If $G$ has exactly two vertices, then to avoid a $U_{2,7}$-minor in $B(G)$, $G$ must be a subgraph of $6K_2$, $5K_2^\ell$, or $4K_2^\circ$. These graphs are, respectively, minors of $O_8$, $F_{10}$, and $K_4^{++}$, a desired outcome. For the remainder of the proof, we now assume that $G$ has at least three vertices. Because $B(G)$ is 3-connected and $G$ has at least three vertices, we now get that $G$ is 2-connected, has no vertices of degree 2, has no free edges, and has no two loops incident to the same vertex. If $G$ is 3-connected, then our result follows from Proposition \ref{P:3Connected}. So for the remainder of the proof we may assume that $G$ is 2-connected but not 3-connected.

Let $\hat G$ be the graph obtained from $G$ by removing all of its loops. Let $T$ be the canonical tree decomposition of $\hat G$. In Case 1 say that $T$ has a 3-connected term, call it $K$, and in Case 2 that every term of $T$ is a cycle or multi-edge.

\noindent{\bf Case 1} By Proposition \ref{P:3Connected}, $K\cong K_4$. If $T$ has another 3-connected term $K'\cong K_4$, then $G$ has a $K_4\oplus_2 K_4$-minor which has a $T_{2,2,3}$-minor, a contradiction. Thus $K$ is the only 3-connected term of $T$. Since $G$ is not 3-connected, there must be a cycle term $C$ adjacent to $K$ in $T$; furthermore, since $G$ has minimum degree 3, there is either a muti-edge term $M$ adjacent to $C$ in $T$ or there is a loop in $G$ incident to one of the internal vertices of $C$. Either possibility, however, creates a $K_4^\ell$-minor in $G$, a contradiction.

\noindent{\bf Case 2} Because $G$ has at least three vertices, $T$ must have a cycle term, $C_s$ which will denote a cycle of length $s\geq 3$. Choose $C_s$ to be the longest such cycle label in $T$ and make it the root of $T$. In Case 2.1 say that the vertex in $T$ corresponding to $C_s$ has degree at least three. In Case 2.2, say that the vertex in $T$ corresponding to $C_s$ has degree two. In Case 2.3, say that the vertex in $T$ corresponding to $C_s$ has degree zero or 1.

\noindent{\bf Case 2.1} If $C_s$ is adjacent to three or more multi-edge terms, then each such multi-edge is $3K_2$ because otherwise we would produce a $T_{3,2,2}$-minor in $G$, a contradiction. Thus $\hat G$ is a minor of $2C_{n}$ which makes $G$ a minor of $2C_{2n}$ (a desired result) unless there is a second cycle term $C_t$ in $T$ which is adjacent to one of the multi-edge terms $M$ in $T$ adjacent to $C_s$. This however, would create a $T_{2,2,2}'$-minor in $\hat G$, a contradiction.

\noindent{\bf Case 2.2} Say that $M_1\cong m_1K_2$ and $M_2\cong m_2K_2$ are the two multiedge terms whose corresponding vertices in $T$ are adjacent to $C_s$. The two edges of $C_s$ into which $M_1$ and $M_2$ are summed are either adjacent or non-adjacent edges. Let these be Cases 2.2.1 and Cases 2.2.2. In both cases $m_1,m_2\geq3$ and $m_1+m_2\leq 8$ because $\hat G$ has no $7K_2$-minor.

\noindent{\bf Case 2.2.1} In Case 2.2.1.1 say that $s=3$, in Case 2.2.1.2 say that $s=4$, and in Case 2.2.1.3 say that $s\geq 5$.

\noindent{\bf Case 2.2.1.1} Since $m_1,m_2\geq3$ and $m_1+m_2\leq 8$, we have that $(m_1,m_2)$ is one of $(4,4)$, $(3,5)$, $(3,4)$, $(3,3)$. Let these be, respectively, Cases 2.2.1.1.1--2.2.1.1.4.

\noindent{\bf Case 2.2.1.1.1} Here $(m_1,m_2)=(4,4)$. First, we claim that $T$ consists of $C_s$, $M_1$, and $M_2$, only. If by way of contradiction, there is another term in $T$, then without loss of generality it must be a cycle term, call it $C$, adjacent to $M_1$. If the vertex corresponding to $C$ in $T$ is a leaf of $T$, then $G$ must have a loop incident to an internal vertex of $C$ and so $G$ has a $6K_2^\ell$-minor, a contradiction. If $C$ is not a leaf of $T$, then it is adjacent to another multi-edge term, call it $M$. This, however, would create a $7K_2$-minor in $\hat G$, a contradiction. Thus $T$ consists of $C_s$, $M_1$, and $M_2$, only and so $\hat G=T_{3,3,1}$; furthermore, if we add a loop anywhere to $\hat G$, then $B(G)$ would have a $6K_2^\ell$-minor, a contradiction. Thus $G=T_{3,3,1}$ which is a minor of $Z_8\dual$, a desired result.

\noindent{\bf Case 2.2.1.1.2} Here $(m_1,m_2)=(3,5)$. In a similar fashion as in Case 2.2.1.1.1, we get that $T$ consists of $C_s$, $M_1$, and $M_2$, only, because otherwise we would be able to construct a $6K_2^\ell$- or $7K_2$-minor in $G$, a contradiction. Thus $\hat G=T_{1,2,4}$ and we cannot add a loop without creating a $6K_2^\ell$-minor. Thus $G=T_{1,2,4}$ and is a minor of $Z_8$, a desired result.

\noindent{\bf Case 2.2.1.1.3} Here $(m_1,m_2)=(3,4)$. If $V(T)=\{C_s,M_1,M_2\}$, then $\hat G=T_{3,2,1}$. Since $G$ has no $5K_2^\circ$-minor, $G$ is therefore a subgraph of $T_{3,2,1}$ along with either a loop at each end of the undoubled edge or a loop at the vertex of degree 5.  Both graphs are minors of $F_{10}$, a desired result.

If $|V(T)|\geq4$, then $T$ has a 3-cycle term $C$ adjacent to either $M_1$ or $M_2$. If $C$ is adjacent to $M_2$, then because $G$ has minimum degree 3, $G$ has as a minor one of graphs 1, 2, 3 of Figure \ref{F:Case22113}. The first two graphs have minors that are bicircular representations of $M(K_{2,3})$, a contradiction. The third graph is $Z_8$. By Lemma \ref{L:Maximality}, $Z_8$ is maximal among 2-connected 4-vertex graphs which do not contain a minor from among $T_{2,2,3}$, $7K_2$, $6K_2^\ell$, and $5K_2^\circ$. Hence if $G$ has four vertices, then $G=Z_8$. If $G$ has five or more vertices, then the tree decomposition $T$ of $\hat G$ must have another 3-cycle term. Again, recall that $G$ has minimum degree 3. Hence, adding on the new 3-cycle term to the tree will replace one edge with at least three edges. This will yield a minor $H$ in $G$ with four vertices which properly contains $Z_8$, a contradiction by Lemma \ref{L:Maximality}.

\begin{figure}[H]
\begin{center}
\includegraphics[page=15,scale=0.5]{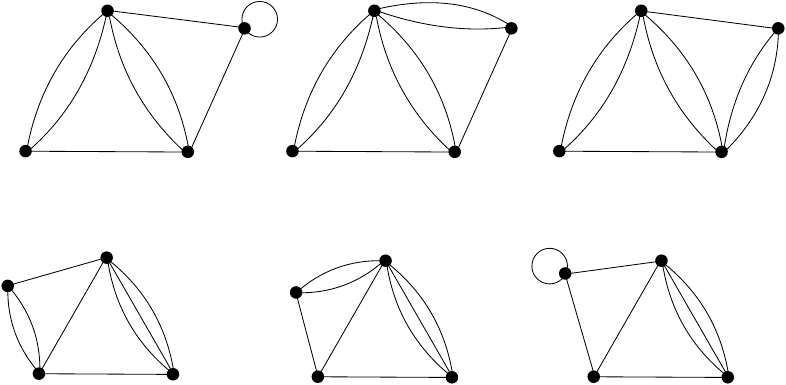}
\end{center}
\caption{Graphs for the proof of Case 2.2.1.1.3} \label{F:Case22113}
\end{figure}

Finally, if $C$ is adjacent to $M_1$, then $G$ contains as a minor one of graphs 4, 5, 6 from Figure \ref{F:Case22113}.  The fifth and sixth graphs of Figure \ref{F:Case22113} both contain minors representing $U_{2,5}''$, a contradiction. The fourth graph is $N_8$. If $G$ has four vertices, then by Lemma \ref{L:Maximality}, $G=N_8$. If $G$ has at least five vertices, then the tree decomposition $T$ of $\hat G$ must have another 3-cycle term. As before, adding in this new 3-cycle term will replace one edge of $G$ with at least three edges and so yields a minor $H$ in $G$ with four vertices which properly contains $N_8$, a contradiction by Lemma \ref{L:Maximality}.

\noindent{\bf Case 2.2.1.1.4} Here $(m_1,m_2)=(3,3)$. If $V(T)=\{C_s,M_1,M_2\}$, then $G$ is a minor of $2C_6$, a desired result. If $T$ has more than three vertices, then vertices $M_1$ and $M_2$ together can have at most two children. If they had three or more, then $G$ would contain as a minor the first graph of Figure \ref{F:Case22114}. Contracting the two unsubdivided edges of this graph yields a minor representing $M(K_{2,3})$, a contradiction.

If we assume that $M_1$ and $M_2$ together have two children, then $G$ contains as a subgraph the graph, call it $K$, obtained from $T_{2,2,1}$ by subdividing two of the four doubled edges. Since $G$ has minimum degree 3, $G$ has a minor $K_0$ which is obtained from $K$ by attaching a link or loop to $K$ to each degree-2 vertex in $K$ with the second endpoint of a link adjacent in $K$ to the first. There are 11 such graphs, all of which have a minor representing $M(K_{2,3})$, a contradiction.

Lastly, assume that $M_1$ and $M_2$ together have just one child. This must be a $3$-cycle term, call it $C$, and say without loss of generality that $C$ is adjacent to $M_1$ in $T$. If $V(T)=\{C_s,M_1,M_2,C\}$, then $G$ has four vertices and must contain as a subgraph the middle graph of Figure \ref{F:Case22114}. Furthermore, $G$ must be obtainable from the middle graph of Figure \ref{F:Case22114} by adding loops. A loop added to vertex 1 would yield a graph with minor representing $M(K_{2,3})$, a contradiction. A loop added to vertex 2 would yield a graph with a $D_4^{\ell\ell}$-minor, a contradiction. Thus $G$ is contained between the two graphs shown in Figure \ref{F:Case22114}, the second graph is a minor of $F_{10}$, a desired outcome.
\begin{figure}[H]
\begin{center}
\includegraphics[page=13,scale=0.8]{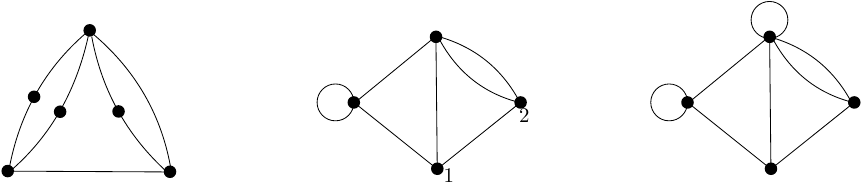}
\end{center}
\caption{Graphs for the proof of Case 2.2.1.1.4} \label{F:Case22114}
\end{figure}
\noindent If $T$ has exactly five vertices, say $V(T)=\{C_s,M_1,M_2,C,M\}$, then the fifth vertex $M$ must be a multi-edge term adjacent to $C$. If $M=tK_2$ for $t\geq 4$, then $\hat G$ contains as a spanning subgraph one of the graphs in the first row of Figure \ref{F:Case22114b}
\begin{figure}[H]
\begin{center}
\includegraphics[page=18,scale=0.6]{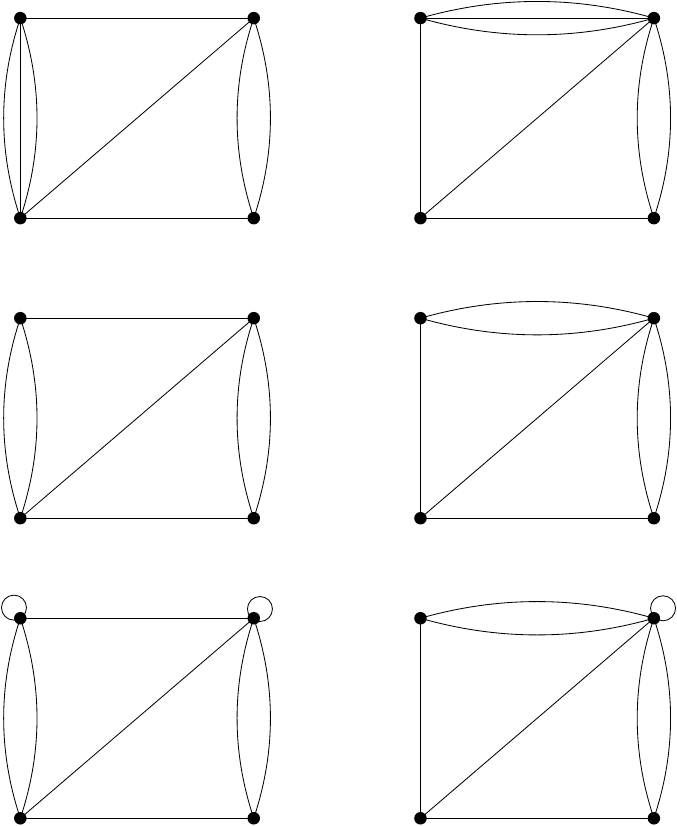}
\end{center}
\caption{More graphs for the proof of Case 2.2.1.1.4} \label{F:Case22114b}
\end{figure} \noindent The second graph contains a minor representing $U_{2,5}''$, a contradiction. The first graph is $N_8$ and because $G$ has four vertices we get that $G=N_8$ by Lemma \ref{L:Maximality}, a desired result. Now we may assume that $M=3K_2$. In this case, $\hat G$ is one of the two graphs in the second row of Figure \ref{F:Case22114b}. We then obtain $G$ from $\hat G$ by adding loops. For the first graph, in order to avoid creating a $U_{2,7}$-minor, we can add at most two loops as shown in the third row of Figure \ref{F:Case22114b}. For the second graph, in order to avoid creating an $M(K_{2,3})$- or $U_{2,5}''$-minor, we can add at most one loop as shown in the third row of the figure. Both of the graphs in the third row are minors of $F_{10}$,  a desired outcome. If $T$ has a sixth vertex, call it $X$, then $X$ is a multi-edge attached to $C$ or is a 3-cycle term attached to $M$. In the former case, $G$ would contain a $T_{2,2,3}$-minor, a contradiction, and so $X$ is a 3-cycle term attached to $M$. In the latter case, $G$ contains one of the graph of Figure \ref{F:Case2211d} as a minor; however, all six of these graphs have minor representing $B(K_{2,3})$, a contradiction.
\begin{figure}[H]
\begin{center}
\includegraphics[page=19,scale=0.6]{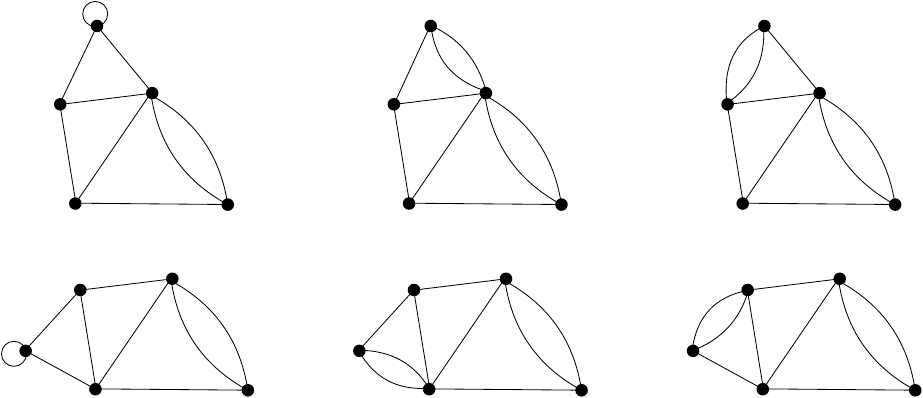}
\end{center}
\caption{More graphs for the proof of Case 2.2.1.1.4} \label{F:Case2211d}
\end{figure}

\noindent{\bf Case 2.2.1.2} Since $M_1$ and $M_2$ are summed into adjacent edges of $C_s$, the vertex of $C_s$ not used by $M_1$ and $M_2$, call it $v$, as a vertex of $G$ and must have a loop incident to it. Again, $m_1,m_2\geq 3$. If $m_1$ or $m_2=4$, then $G$ contains the graph on the left of Figure \ref{F:Case2212} as a minor and this graph has a subgraph representing $U_{2,5}''$, a contradiction.
\begin{figure}[H]
\begin{center}
\includegraphics[page=20,scale=0.4]{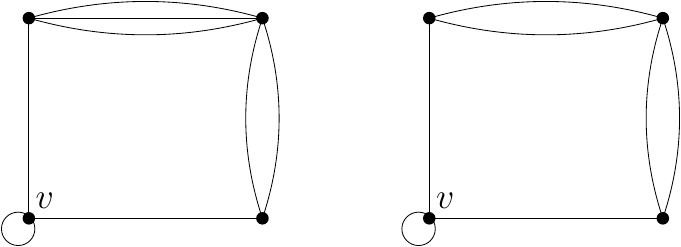}
\end{center}
\caption{Graphs for the proof of Case 2.2.1.2.} \label{F:Case2212}
\end{figure} \noindent Thus $m_1=m_2=3$. If $V(T)=\{C_s,M_1,M_2\}$, then $G$ is obtained from the graph on the right of Figure \ref{F:Case2212} by adding loops. Thus $G$ is a minor of $2C_8$, a desired result. If $T$ has a fourth vertex, then it must be a cycle term, call it $C$, which is, without loss of generality, adjacent to $M_1$. Thus $G$ contains one of the graphs of Figure \ref{F:Case2212} as a minor. Each of these graphs. however, contains a minor repreenting $M(K_{2,3})$, a contradiction. \begin{figure}[H]
\begin{center}
\includegraphics[page=2,scale=0.4]{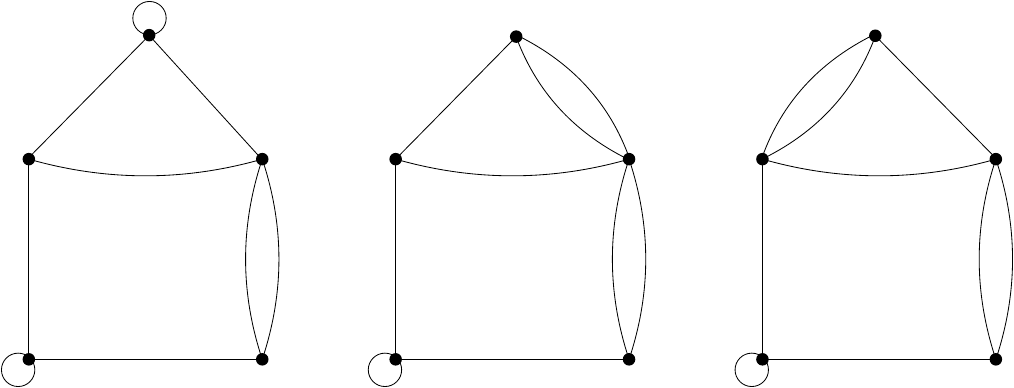}
\end{center}
\caption{Graphs for the proof of Case 2.2.1.2} \label{F:Case2212}
\end{figure} \noindent

\noindent{\bf Case 2.2.1.3} In a very similar fashion as in Case 2.2.1.2 we get that $G$ is a minor of $2C_n$.

\noindent{\bf Case 2.2.2} We cannot have that $s=3$. If $s\geq 5$, then in a very similar fashion as in Case 2.2.1.2, we get that $G$ is a minor of $2C_n$, a desired result. So it remains to consider the case in which $s=4$. Again, we have that $(m_1,m_2)$ is one of $(4,4)$, $(3,5)$, $(3,4)$, and $(3,3)$. Let these be, respectively Cases 2.2.2.1 -- 2.2.2.4.

\noindent{\bf Case 2.2.2.1} Here $(m_1,m_2)=(4,4)$. Thus $G$ contains $Z_8\dual$ as minor. If $G$ contains $Z_8\dual$ as a subgraph, then again, $G=Z_8\dual$ by Lemma \ref{L:Maximality}. If $G$ does not contain a $Z_8\dual$-subgraph, then the tree decomposition of $\hat G$ has vertex set containing $\{C_s,M_2,M_2,C\}$ in which $C$ is a cycle term summed onto $M_1$. Since $G$ has no vertices of degree 2, then $G$ contains a minor obtained from $Z_8\dual$ by adding an edge with both endpoints in $Z_8\dual$. Again, Lemma \ref{L:Maximality} implies that $G=\bb Z_8\dual$.

\noindent{\bf Case 2.2.2.2} Say that $(m_1,m_2)=(3,5)$. In a similar fashion as with $(m_1,m_2)=(4,4)$ we get that $G\cong O_8$, a desired result.

\noindent{\bf Case 2.2.2.3} Say that $(m_1,m_2)=(3,4)$. If $V(T)=\{C_s,M_1,M_2\}$, then $\hat G$ is the first graph shown in Figure \ref{F:Case2223}, call if $O$. If we add a loop to one of the top two vertices of $O$ and to one of the bottom two vertices of $O$, then the resulting graph has a minor representing $U_{2,7}$, a contradiction. If we add loops to both of the top vertices, then we obtained the second graph of Figure \ref{F:Case2223} which is a minor of $F_{10}$, a desired result. If $T$ has a fourth vertex, then it is a cycle term attached to either $M_1$ or $M_2$. Therefore $G$ will contain as a minor one of the last four graphs of Figure \ref{F:Case2223}. The third and fourth graphs both contain a bicircular representation of $U_{2,5}''$ as a minor, a contradiction. The fifth and sixth graphs both contain a bicircular representation of $M(K_{2,3})$ as a minor, again a contradiction.

\begin{figure}[H]
\begin{center}
\includegraphics[page=21,scale=0.6]{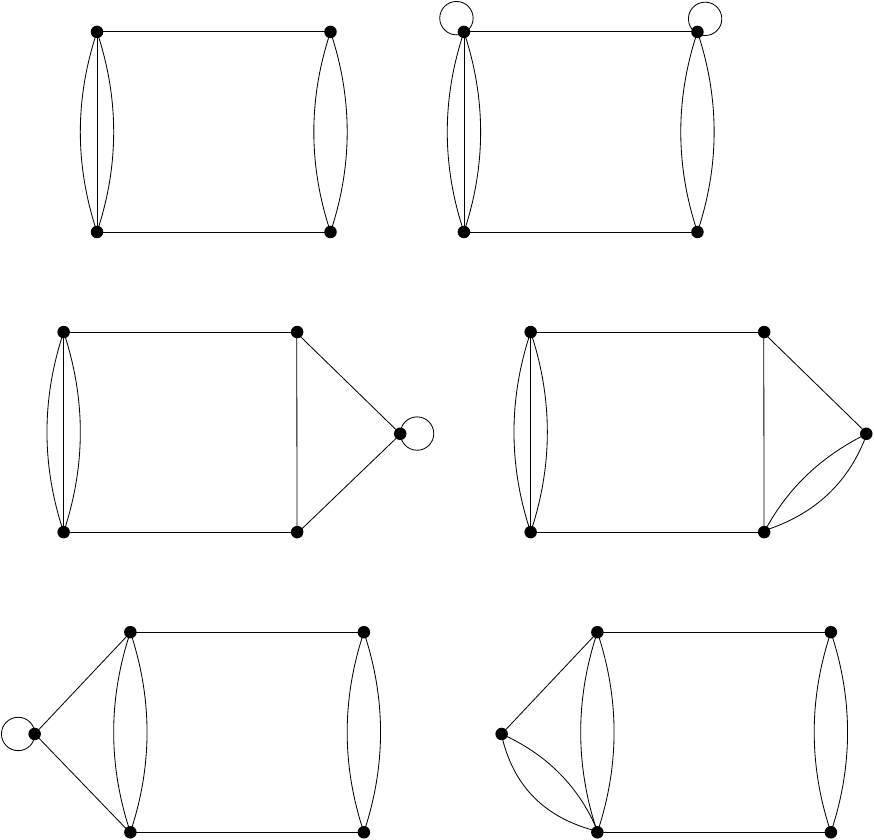}
\end{center}
\caption{Graphs for the proof of Case 2.2.2.3} \label{F:Case2223}
\end{figure}

\noindent{\bf Case 2.2.2.4} Say that $(m_1,m_2)=(3,3)$. If $V(T)=\{C_s,M_1,M_2\}$, then $\hat G$ is $C_4^{++}$ and so $G$ is a minor of $2C_8$, a desired result. If $V(T)$ has a fourth vertex, then it is a cycle term, call it $C$, and say without loss of generality that $C$ is adjacent to $M_1$ in $T$. If $C$ is a leaf of $T$, then $G$ contains as a minor the first graph of Figure \ref{F:Case2224}. This graph, however, contains a $D_4^{\ell\ell}$-minor, a contradiction. Thus $V(T)$ must contain a fifth vertex, call it $X$, that is adjacent to $C$, which makes $X$ a multi-edge term. Suppose now that $V(T)=\{C_s,M_1,M_2,C,X\}$. If $X=tK_2$ for $t\geq 4$, then $G$ contains as a minor the second graph of Figure \ref{F:Case2224} which contains a bicircular representation of $U_{2,5}''$ as a minor, a contradiction. Thus $X=3K_2$ and so $\hat G$ is the third graph of Figure \ref{F:Case2224}, call it $F$. If we add a loop to $F$ at vertex 1, then the resulting graph has a $D_4^{\ell\ell}$-minor, a contradiction. If we add a loop to $F$ at vertex 2, then the resulting graph contains a bicircular representation for $M(K_{2,3})$ as a minor, a contradiction. If we add a loop to vertex $3$, then the resulting graph contains a bicircular representation of $U_{2,5}''$ as a minor, again a contradiction. Thus $G$ is contained between the third graph of the figure and the graph obtained by adding loops to both of the unnumbered vertices. This latter graph is $F_{10}$, a desired outcome.\begin{figure}[H]
\begin{center}
\includegraphics[page=22,scale=0.6]{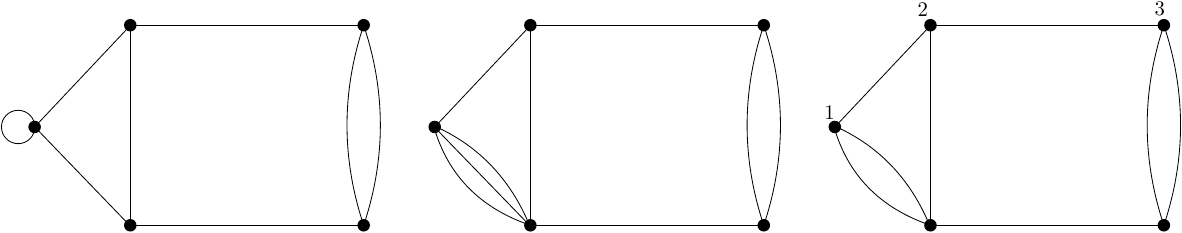}
\end{center}
\caption{More graphs for the proof of Case 2.2.2.4} \label{F:Case2224}
\end{figure} So now suppose that $V(T)$ properly contains $\{C_s,M_1,M_2,C,X\}$ and so has a sixth vertex, call it $Y$. The vertex $Y$ is either a multi-edge term adjacent to $C$ or a cycle term adjacent to one of $M_1$, $M_2$, and $X$. The reader can verify all of the outcomes in these four cases. One, if $Y$ is adjacent to $C$, then $G$ contains a $T_{2,2,3}$-minor, a contradiction. Two, if $Y$ is adjacent to $X$, then $G$ contains a representation of $M(K_{2,3})$, a contradiction. Three, if $Y$ is adjacent to $M_2$, then $G$ either contains a representation of $M(K_{2,3})$ or $U_{2,5}''$, a contradiction. Four, if $Y$ is adjacent to $M_1$, then $G$ contains a representation of $M(K_{2,3})$ as a minor, a contradiction.

\noindent{\bf Case 2.3} If $C_s$ has degree zero in $T$, then $\hat G=C_s$ and so $G$ is obtained from $C_s$ by adding a loop to each vertex. Thus $G$ is a minor of $2C_{2s}$, a desired outcome. If $C_s$ has degree 1 in $T$, then we split the remainder of this case into two subcases. In Case 2.3.1, say that $s\geq 4$ and in Case 2.3.2 say $s=3$. In each case, let $M$ be the multiedge term whose corresponding vertex in $T$ is adjacent to $C_s$.

\begin{figure}[H]
\begin{center}
\includegraphics[page=23,scale=0.6]{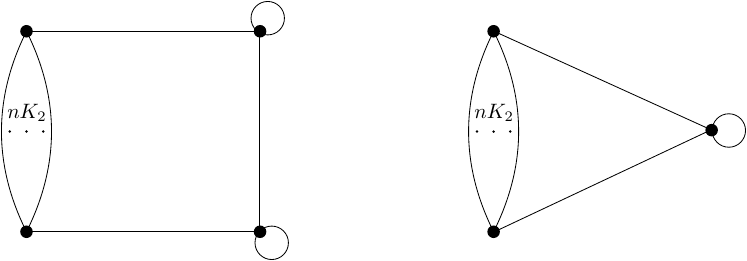}
\end{center}
\caption{Graphs for the proof of Case 2.3.} \label{F:Case23}
\end{figure}

\noindent{\bf Case 2.3.1} In this case $\hat G$ contains as a minor the first graph shown in Figure \ref{F:Case23}. If $n\geq 3$, then the graph shown contains a bicircular representation of $U_{2,5}''$-minor, a contradiction. So now suppose that $n=2$. If $V(T)=\{C_s,M\}$, then $G$ is a minor of $2C_{2s}$, a desired outcome. If $T$ contains a third vertex, then this is a cycle term, call it $C$, which is adjacent to $M$. In this case $G$ contains as a minor one of the two graphs shown in Figure \ref{F:Case231}. Both of these graphs, however, contain a bicircular representation of $M(K_{2,3}$), a contradiction.

\begin{figure}[H]
\begin{center}
\includegraphics[page=3,scale=0.6]{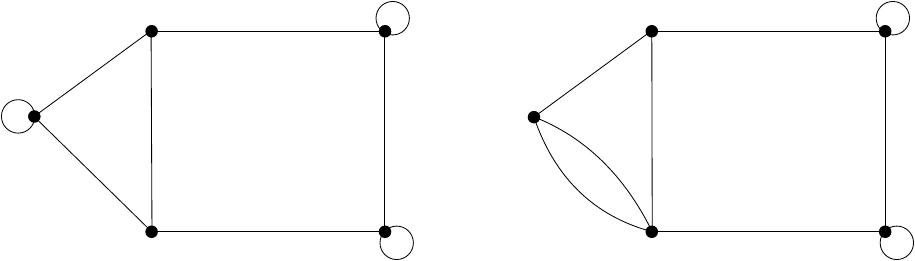}
\end{center}
\caption{Graphs for the proof of Case 2.3.1.} \label{F:Case231}
\end{figure}

\noindent{\bf Case 2.3.2} In this case, $\hat G$ is obtained from the second graph of Figure \ref{F:Case23} by removing its loop. If $n=2$, then $G$ is a minor of $2C_6$, a desired outcome. If $n=3$, then the second graph of the figure with one additional loop added is a minor of $F_{10}$, a desired outcome. If a third loop is added, then the graph obtained contains a representation of $U_{2,6}'$ as a minor, a contradiction. If $n=4$, then the second graph of the Figure \ref{F:Case23} is a minor of $O_8$, a desired outcome. If one loop is added, then the graph contains a bicircular representation of $U_{2,7}$ as a minor, a contradiction.
\end{proof}

\bibliographystyle{amsplain}
\bibliography{Cobicircular_bibfile}

\end{document}